\documentclass[12pt,a4paper,twoside]{amsart}
\usepackage{a4}

\usepackage[cp1251]{inputenc}
\usepackage[T2A]{fontenc}

\usepackage[english]{babel}
\usepackage{amsmath, amsfonts, amssymb, amsthm}

\usepackage{pgf,tikz}
\usetikzlibrary{arrows}

\newtheorem{theorem}{Theorem}
\newtheorem{lemma}{Lemma}[section]

\newtheorem{corollary}{Corollary}[section]
\newtheorem{definition}{Definition}[section]

\theoremstyle{definition}
\newtheorem*{ack}{Acknowledgments}

\theoremstyle{remark}
\newtheorem{remark}{Remark}[section]

\numberwithin{equation}{section}

\DeclareMathOperator{\arccot}{arccot}
\DeclareMathOperator{\arccoth}{arccoth}
\DeclareMathOperator{\arccosh}{arccosh}
\DeclareMathOperator{\normc}{n}
\DeclareMathOperator{\dist}{dist}

\title[Some sharp estimates for convex hypersurfaces]{Some sharp estimates for convex hypersurfaces of pinched normal curvature}
\author{Kostiantyn Drach}

\address{Geometry Department \\ V.N. Karazin Kharkiv National University\\ Svobody Sq. 4, 61022, Kharkiv \\ Ukraine}

\address{Department of Mathematical Analysis and Optimization \\ Sumy State University \\Rimskogo~- Korsakova str. 2, 40007, Sumy \\ Ukraine}
	
\email{drach@karazin.ua, kostya.drach@gmail.com}



\begin{document}

\maketitle

\begin{abstract}
For a convex domain $D$ bounded by the hypersurface $\partial D$ in a space of constant curvature we give sharp bounds on the width $R-r$ of a spherical shell with radii $R$ and $r$ that can enclose $\partial D$, provided that normal curvatures of $\partial D$ are pinched by two positive constants. Furthermore, in the Euclidean case we also present sharp estimates for the quotient $R/r$. From the obtained estimates we derive stability results for almost umbilical hypersurfaces in the constant curvature spaces.
\end{abstract}


\section{Preliminaries and the main results}


In~\cite{BorMiq1} A.~Borisenko and V.~Miquel proved that a closed hypersurface with normal curvatures $k_{\normc}$ satisfying the inequality $k_{\normc} \geqslant 1$  in the Lobachevsky space $\mathbb H^m(-1)$ can be put into a spherical shell between two concentric spheres of radii $R$ and $r$ such that the \textit{width} $R-r$ of the shell satisfies $R-r \leqslant \ln 2$. A similar estimate holds in Hadamard manifolds (see~\cite{BorMiq2}). In~\cite{BorDr2} these results were extended to Riemannian manifolds of constant-signed sectional curvatures and hypersurfaces with normal curvatures bounded below.

In the present paper we refine some results from~\cite{BorDr2}. For this purpose we consider hypersurfaces with normal curvatures at any point and in any direction pinched by two positive constants. Such restriction allows us to obtain sharper estimates for the width $R-r$ than in~\cite{BorDr2} (see~\cite{RSch} for some related results). Furthermore, for surfaces of pinched normal curvature we are able to derive an upper bound on the quotient $R/r$, which, in contrast, can be arbitrarily large for a hypersurface with normal curvatures just bounded below.

Besides deriving sharp estimates for $R-r$ and $R/r$, the normal curvature pinching condition enables us to obtain accurate stability results for the so-called \textit{almost umbilical} hypersurfaces, that is hypersurfaces whose normal curvatures are pinched between $\kappa$ and $(1+\varepsilon) \kappa$, in constant curvature spaces. These results are the best possible for such a uniform pinching (see~\cite{Gr} for a weaker Euclidean version of the result; for the results on stability of almost umbilical hypersurfaces with pointwise principal curvatures pinching see~\cite[p. 493]{Pog}~\cite{Lei}; for stability with some integral pinching conditions see \cite{Dis, JSch}).

Let us denote by $\mathbb M^m (c)$ with $m \geqslant 2$ a complete simply connected $m$-dimensional Riemannian manifold of constant sectional curvature equal to $c$. In order to state the main results, we need the following definition. 

\begin{definition}
A hypersurface $F \subset \mathbb M^{m}(c)$  is said to be \emph{$\kappa_1,\kappa_2$-convex}  (with $\kappa_2 \geqslant \kappa_1$, and for $c = 0$ we assume that  $\kappa_1 >0$, for $c > 0$ we assume that $\kappa_1 \geqslant 0$, and for $c < 0$ we assume that $\kappa_1 > \sqrt{-c}$), if for any point $P \in F$ there exist two nested geodesic spheres $S_2 \subset S_1 \subset \mathbb M^{m} (c)$ of constant normal curvatures equal to, respectively, $\kappa_1$ and $\kappa_2$, passing through $P$ such that locally near $P$ the hypersurface $F$ lies inside $S_1$ and outside $S_2$.  
\end{definition}

Observe that $\kappa_1,\kappa_2$-convex hypersurfaces are, in particular, $\kappa_1$-convex (see~\cite{BorMiq2} and~\cite{BorDr2}). The usual notion of convexity can be viewed as $0,\infty$-convexity assuming that $S_2$ in this case is a point and $S_1$ is a totally geodesic hyperplane.

We should note that for $C^r$-smooth hypersurfaces with $r \geqslant 2$ the property of being $\kappa_1,\kappa_2$-convex is equivalent to that all its normal curvatures $k_{\normc}$ (or, equivalently, second fundamental forms) with respect to the inner normal vector field are $\kappa_1,\kappa_2$-pinched, that is $\kappa_1 \leqslant k_{\normc} \leqslant \kappa_2$. In general, since some small neighborhood of any point $P$ on a $\kappa_1,\kappa_2$-convex hypersurface $F$ lies between two tangent at $P$ geodesic spheres, we have that $F$ is $C^{1,1}$-smooth (see~\cite{Wal}). Therefore, by generalized Rademacher's theorem (see again~\cite{Wal}) at almost all its points a $\kappa_1,\kappa_2$-convex hypersurfaces has well-defined normal curvatures (second fundamental forms) satisfying the inequality shown above.

A closed domain $D \subset \mathbb M^{m}(c)$ is called \textit{$\kappa_1,\kappa_2$-convex} if its boundary $\partial D$ is a $\kappa_1,\kappa_2$-convex hypersurface. Such domains are homeomorphic to geodesic balls of the corresponding spaces. 

We recall that for $\kappa_1,\kappa_2$-convex domains well-known \textit{Blaschke's rolling theorem} holds (see~\cite{Bla}, \cite{Kar}, and~\cite{How} for a smooth case, and~\cite{M},~\cite{BorDr3} for a general case). More precisely, it states the following. Suppose $D \subset \mathbb M^m(c)$ is a $\kappa_1,\kappa_2$-convex domain; for any point $P \in \partial D$ let $S_1$ and $S_2$ be two nested spheres of normal curvature equal to, respectively, $\kappa_1$ and $\kappa_2$, and that are tangent to $\partial D$ at $P$; then $B_2 \subseteq D \subseteq B_1$, where $B_i$ is the closed geodesic ball bounded by $S_i$, $i \in \{1,2\}$. 

We are now ready to state the main results of the paper.

\begin{theorem}
\label{th1}
If $D \subset \mathbb M^{m}(c)$ is a $\kappa_1,\kappa_2$-convex domain, then the hypersurface $\partial D$ can be put into a spherical shell between two concentric spheres of radii $R$ and $r$ (with $R \geqslant r$) such that
\begin{enumerate}
\item
for $c = 0$,
\begin{equation}
\label{euclweq}
R-r \leqslant \left(\sqrt{2} - 1\right) \left(R_1 - R_2\right);
\end{equation}
\item
for $c = k^2$ with $k > 0$,
\begin{equation}
\label{spweq}
R-r \leqslant \frac{2}{k}\arccos \sqrt{\cos\left(k (R_1 - R_2)\right)} - (R_1 - R_2);
\end{equation}
\item
for $c = -k^2$ with $k > 0$,
\begin{equation}
\label{hypweq}
R-r \leqslant \frac{2}{k}\arccosh \sqrt{\cosh\left(k (R_1 - R_2)\right)} - (R_1 - R_2),
\end{equation}
\end{enumerate}
where $R_1$ and $R_2$ are the radii of circles with geodesic curvatures equal to, respectively, $\kappa_1$ and $\kappa_2$, and lying in the corresponding $2$-planes $\mathbb M^2(c)$. 

Moreover, these estimates are sharp.
\end{theorem}

\begin{remark}
It is known that $R_i = 1/\kappa_i$ for $c=0$, $R_i=1/k \arccot (\kappa_i/k)$ for $c=k^2$, and $R_i = 1/k \arccoth (\kappa_i/k)$ for $c=-k^2$, $i \in \{1,2\}$.
\end{remark}
\begin{remark}
As $\kappa_2 \to \infty$, estimates~(\ref{euclweq})~-- (\ref{hypweq}) tend to the corresponding estimates in the spaces of constant curvature from~\cite{BorDr2}.
\end{remark}
\begin{remark}
By sharpness of the inequalities above and below we mean that the shown bounds are attained by so-called \textit{rounded $\kappa_1,\kappa_2$-convex spindle-shaped surfaces} (see Fig.~\ref{pic1}), which we describe in details in the next section. 
\end{remark}

In the Euclidean case we can give even more interesting estimate for the quotient $R/r$.

 \begin{theorem}
\label{th2}
If $D \subset \mathbb E^{m}$ is a $\kappa_1,\kappa_2$-convex domain in the Euclidean space, then the hypersurface $\partial D$ can be put into a spherical shell between two concentric spheres of radii $R$ and $r$ (with $R \geqslant r$) such that
\begin{equation}
\label{qeq}
\frac{R}{r} \leqslant \frac{\sqrt{\frac{\kappa_2}{\kappa_1}}+\sqrt{2}}{\sqrt{\frac{\kappa_1}{\kappa_2}}+\sqrt{2}}.
\end{equation}

Moreover, this estimate is sharp.
\end{theorem}

\begin{corollary}
\label{cor1}
In the condition of Theorem~\ref{th2}, it is true that
$$
\frac{R}{r} \leqslant \frac{\kappa_2}{\kappa_1}.
$$ 
\end{corollary}

\begin{remark}
In the theorems and the corollary above, as we will see from the proofs, the center of the shell can be chosen to coincide with the center of the inscribe ball for the domain $D$. Compare this to ~\cite{RSch}, where Corollary~\ref{cor1} is proved for the center of the shell being one of the curvature centroids of $\partial D$. 
\end{remark}

\begin{remark}
We note that if in the above $\kappa_1 = \kappa_2$, then from all estimates~(\ref{euclweq})~--~(\ref{qeq}) it follows that $R = r$, and thus the domain $D$  is a geodesic ball of the corresponding space. 
\end{remark}

Theorems~\ref{th1} and~\ref{th2} are based on the following result, which is useful by itself. It gives the sharp upper bound on the outer radius $R$ of the spherical shell in terms of the inner radius $r$ of that shell.

\begin{theorem}
\label{Rrestth}
If $D \subset \mathbb M^{m}(c)$ is a $\kappa_1,\kappa_2$-convex domain, then the hypersurface $\partial D$ can be put into a spherical shell between two concentric spheres of radii $R$ and $r$ (with $R\geqslant r$) such that
\begin{enumerate}
\item
for $c=0$,
\begin{equation}
\label{Rresteq1}
R \leqslant \sqrt{\left(R_1 - R_2\right)^2 - \left(R_1 - r\right)^2} + R_2;
\end{equation}
\item
for $c = k^2$ with $k>0$,
\begin{equation}
\label{Rresteq2}
R \leqslant \frac{1}{k}\arccos \frac{\cos \left(k (R_1 - R_2)\right)}{\cos \left(k (R_1 - r)\right)} + R_2;
\end{equation}
\item
for $c = -k^2$ with $k>0$,
\begin{equation}
\label{Rresteq3}
R \leqslant \frac{1}{k}\arccosh \frac{\cosh \left(k(R_1 - R_2)\right)}{\cosh \left(k(R_1 - r)\right)} + R_2,
\end{equation}
\end{enumerate}
where $R_1$ and $R_2$ are the radii of circles with geodesic curvature equal to, respectively, $\kappa_1$ and $\kappa_2$, and lying in the corresponding $2$-planes $\mathbb M^2(c)$.

Moreover, estimates~(\ref{Rresteq1})~-- (\ref{Rresteq3}) are sharp.
\end{theorem}

As it was said earlier, the results of Theorems~\ref{th1} and~\ref{th2} allow us to derive accurate stability results for almost umbilical hypersurfaces. In particular, we prove the following

\begin{theorem}
\label{stabilitythm}
If $D \subset \mathbb M^m(c)$ is a $\kappa,(1+\varepsilon) \kappa$-convex domain for some $\varepsilon \geqslant 0$, then the hypersurface $\partial D$ can be put into a spherical shell between two concentric spheres with radii $R$ and $r$ (with $R \geqslant r$) such that $$R - r <  C(\kappa,c)\cdot\varepsilon,$$
with $C(\kappa,c) = \frac{\kappa}{\kappa^2+c} \left(\sqrt{2}-1\right) $, and this constant is the best possible.

Moreover, for $c=0$, the hypersurface $\partial D$ can be put into a spherical shell between two concentric spheres with radii $R$ and $r$ ($R \geqslant r$) such that $$\frac{R}{r}-1 < C \cdot \varepsilon$$ with $C = \sqrt{2}-1$, and this constant is the best possible.
\end{theorem}


\section{Proofs of the main results}


We are going to prove the theorems above by using a comparison argument. Let us introduce an object to compare with.

In $\mathbb M^{m}(c)$ let us consider a spindle-shaped $\kappa_1$-convex hypersurface, that is a hypersurface obtained by rotating a smaller circular arc $PQ$ of geodesic curvature equal to $\kappa_1$ (see~\cite{BorDr2}). Such surface have two vertexes $P$ and $Q$ where its normal curvatures blow up. After smoothing these vertexes  using two spherical caps of normal curvature equal to $\kappa_2$ whose centers lie on the geodesic line $PQ$, we obtain a convex $C^{1,1}$-smooth hypersurface (see Fig.~\ref{pic1}). We will call such surfaces~\emph{rounded $\kappa_1,\kappa_2$-convex spindle-shaped hypersurfaces}. 


\begin{figure}[h]
\begin{center}
\begin{tikzpicture}[line cap=round,line join=round,>=triangle 45,x=1.0cm,y=1.0cm,scale=1.6]
\clip(-2.26,-2.58) rectangle (2.09,0.71);
\draw [shift={(-0.03,-2.41)},line width=1.2pt]  plot[domain=1.03:2.11,variable=\t]({1*2.8*cos(\t r)+0*2.8*sin(\t r)},{0*2.8*cos(\t r)+1*2.8*sin(\t r)});
\draw [shift={(-1.24,-0.4)},line width=1.2pt,color=black,fill=black,fill opacity=0.1]  plot[domain=2.11:4.17,variable=\t]({1*0.45*cos(\t r)+0*0.45*sin(\t r)},{0*0.45*cos(\t r)+1*0.45*sin(\t r)});
\draw [shift={(-0.03,1.61)},line width=1.2pt]  plot[domain=4.17:5.25,variable=\t]({1*2.8*cos(\t r)+0*2.8*sin(\t r)},{0*2.8*cos(\t r)+1*2.8*sin(\t r)});
\draw [shift={(1.17,-0.4)},line width=1.2pt,color=black,fill=black,fill opacity=0.1]  plot[domain=-1.03:1.03,variable=\t]({1*0.45*cos(\t r)+0*0.45*sin(\t r)},{0*0.45*cos(\t r)+1*0.45*sin(\t r)});
\draw [shift={(-0.03,-0.4)}] plot[domain=0:3.14,variable=\t]({0*0.78*cos(\t r)+-1*0.17*sin(\t r)},{1*0.78*cos(\t r)+0*0.17*sin(\t r)});
\draw [shift={(-0.03,-0.4)},dash pattern=on 1pt off 1pt]  plot[domain=-3.14:0,variable=\t]({0*0.78*cos(\t r)+-1*0.17*sin(\t r)},{1*0.78*cos(\t r)+0*0.17*sin(\t r)});
\draw [shift={(-1.47,-0.4)},color=black] plot[domain=0:3.14,variable=\t]({0*0.39*cos(\t r)+-1*0.08*sin(\t r)},{1*0.39*cos(\t r)+0*0.08*sin(\t r)});
\draw [shift={(-1.47,-0.4)},dash pattern=on 1pt off 1pt,fill=black,fill opacity=0.1,color=black]  plot[domain=-3.14:0,variable=\t]({0*0.39*cos(\t r)+-1*0.08*sin(\t r)},{1*0.39*cos(\t r)+0*0.08*sin(\t r)});
\draw [shift={(1.4,-0.4)},fill=black,fill opacity=0.1,color=black]  plot[domain=-3.14:0,variable=\t]({0*0.39*cos(\t r)+1*0.08*sin(\t r)},{1*0.39*cos(\t r)+0*0.08*sin(\t r)});
\draw [shift={(1.4,-0.4)},line width=0.4pt,dash pattern=on 1pt off 1pt,color=black]  plot[domain=0:3.14,variable=\t]({0*0.39*cos(\t r)+1*0.08*sin(\t r)},{1*0.39*cos(\t r)+0*0.08*sin(\t r)});
\draw [shift={(-1.24,-0.4)},dash pattern=on 1pt off 1pt]  plot[domain=-2.11:2.11,variable=\t]({1*0.45*cos(\t r)+0*0.45*sin(\t r)},{0*0.45*cos(\t r)+1*0.45*sin(\t r)});
\draw [shift={(-0.03,1.61)},dash pattern=on 1pt off 1pt]  plot[domain=3.95:4.17,variable=\t]({1*2.8*cos(\t r)+0*2.8*sin(\t r)},{0*2.8*cos(\t r)+1*2.8*sin(\t r)});
\draw [shift={(-0.03,-2.41)},dash pattern=on 1pt off 1pt]  plot[domain=2.11:2.34,variable=\t]({1*2.8*cos(\t r)+0*2.8*sin(\t r)},{0*2.8*cos(\t r)+1*2.8*sin(\t r)});
\draw [shift={(-0.03,-2.41)},dash pattern=on 1pt off 1pt]  plot[domain=0.8:1.03,variable=\t]({1*2.8*cos(\t r)+0*2.8*sin(\t r)},{0*2.8*cos(\t r)+1*2.8*sin(\t r)});
\draw [shift={(-0.03,1.61)},dash pattern=on 1pt off 1pt]  plot[domain=5.25:5.48,variable=\t]({1*2.8*cos(\t r)+0*2.8*sin(\t r)},{0*2.8*cos(\t r)+1*2.8*sin(\t r)});
\draw [dotted] (-1.97,-0.4)-- (-0.03,-2.41);
\draw [dotted] (-0.03,-2.41)-- (1.91,-0.4);
\draw [dotted] (-1.47,-0.01)-- (-1.24,-0.4);
\draw [dotted] (-1.24,-0.4)-- (-1.47,-0.79);
\draw (0.59,0.63) node[anchor=north west] {$\kappa_1$};
\draw (0.59,-0.66) node[anchor=north west] {$\kappa_1$};
\draw (-2,0.15) node[anchor=north west] {$\kappa_2$};
\draw (1.53,0.15) node[anchor=north west] {$\kappa_2$};
\draw (-1.42,0.05) node[anchor=north west] {$R_2$};
\draw (-0.79,-1.46) node[anchor=north west] {$R_1$};
\fill [color=black] (-0.03,-0.4) circle (0.5pt)  node[below] {$\tilde O$};
\fill [color=black] (-1.97,-0.4) circle (0.5pt);
\draw (-1.97,-0.4) node[below,xshift=-3pt] {$P$};
\fill [color=black] (-1.24,-0.4) circle (0.5pt);
\fill [color=black] (-0.03,-2.41) circle (0.5pt);
\draw [fill=black] (1.91,-0.4) circle (0.5pt) node[below, xshift=1pt] {$Q$};
\end{tikzpicture}
\caption{Rounded $\kappa_1,\kappa_2$-convex spindle-shaped hypersurface}
\label{pic1}
\end{center}
\end{figure}
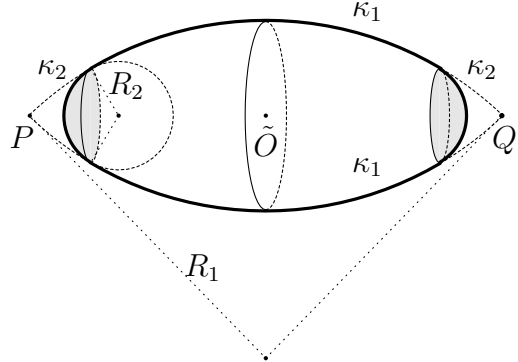


Since the constructed above rounded $\kappa_1,\kappa_2$-convex spindle-shaped hypersurface is centrally symmetric with respect to the midpoint $\tilde O$ of the geodesic segment $PQ$ (see Fig.~\ref{pic1}), the center of its inscribed ball coincides with $\tilde O$, and the radius $\tilde r$ of this ball is equal to the distance from $\tilde O$ to the hypersurface. Note that $\tilde r$ varies from $R_2$ to $R_1$, and for each $\tilde r \in [R_2, R_1]$ there exist a unique rounded $\kappa_1,\kappa_2$-convex spindle-shaped hypersurface with $\tilde r$ being the radius of its inscribed ball. Thus we get a one-parametric family of such surfaces.  

Now we can prove the key comparison lemma.

\begin{lemma}
\label{REst}
Let $D \subset \mathbb M^{m}(c)$ be a closed $\kappa_1, \kappa_2$-convex domain, $r$ be the radius of the inscribe sphere for $D$ with center at a point $O$. Let $\tilde F \subset \mathbb M^{m}(c)$ be a rounded $\kappa_1, \kappa_2$-convex spindle-shaped hypersurface, and let $\tilde r$, $\tilde R$ be the radii of its inscribe and circumscribe spheres. If
$$\tilde r =  r,$$ 
then
\begin{equation}
\label{maxdistrel}
\max \dist\left(O, \partial D\right) \leqslant \tilde R.
\end{equation}

Moreover, this bound is sharp.
\end{lemma}

\begin{proof}
We will argue by contradiction. Suppose that~(\ref{maxdistrel}) is not true, and the inverse inequality 
\begin{equation}
\label{contr1}
\max \dist\left(O, \partial D\right) > \tilde R
\end{equation}
holds. For simplicity, we denote the hypersurface $\partial D$ by $F$.

Let $M \in F$ be a point such that $\max \dist (O, F) = |OM|$ (here and below $|\cdot|$ denotes the distance between two points);  then from~(\ref{contr1}) it follows that on the geodesic segment $OM$ there exists a point $A$ such that 
\begin{equation}
\label{contr2}
|OA| = \tilde R < |OM|.
\end{equation}

Assume that the rounded $\kappa_1,\kappa_2$-convex spindle-shaped hypersurface $\tilde F$ is centered at $O$, and its rotational axis coincides with the geodesic line $OA$. Then $A \in \tilde F$. 

Since the point $M$ is a point for which the maximal distance from $O$ is attained, we have that a totally geodesic hyperplane which touches $F$ at $M$ is perpendicular to the geodesic line $OM$. Therefore, if $\omega_2 \subseteq D$ is a sphere with the center at a point $O_2$, and with normal curvature equal to $\kappa_2$ that touches from inside the hypersurface $F$ at $M$ (such a sphere exists by Blaschke's rolling theorem), then $O_2$ lies on the segment $OM$. We note that from~(\ref{contr2}) it follows
\begin{equation}
\label{contr3}
|OO_2| > \tilde R - R_2,
\end{equation}
where $R_2$ is the radius of $\omega_2$.

Let us denote the inscribe sphere  for $F$  by $\omega$. Since the hypersurface $F$ is $\kappa_1, \kappa_2$-convex, then by Blaschke's rolling theorem $F$ lies in a ball of radius $R_1$ (we remind that $R_1$ is the radius of a sphere of normal curvature equal to $\kappa_1$). Hence, both $\omega$ and $\omega_2$ lies in this ball too. Now we show that there is a sphere of radius $R_1$ that touches externally both $\omega$ and $\omega_2$  simultaneously. 

Denote $|OO_2|$ by $d$. It is clear that the sphere mentioned above exists if and only if the three numbers $d$, $R_1 - r$, and $R_1 - R_2$ satisfy the triangle inequality. Let us check this:
\begin{enumerate}
\item
$(R_1 - r) + (R_1 - R_2) = 2R_1 - (r + R_2) > d$, since $\omega$ and $\omega_2$ lie in a ball of radius $R_1$;
\item
$(R_1 - R_2) + d > R_1 - r$, because $\omega$ is the inscribe sphere and thus either $\omega \equiv \omega_2$ and $r = R_1 = R_2$ that is a trivial case when $F$ is a sphere, or $\omega_2$ touches $\omega$ from inside, which is again a trivial case when $F$ is a sphere, or $\omega_2$ cannot lie entirely inside $\omega$, hence $r + d < R_2$.
\item
$(R_1-r) + d > R_1 - R_2$, which is obviously true. 
\end{enumerate}

By rotational symmetry, along with a single sphere of radius $R_1$ there exists a family of spheres of the same radius that touches simultaneously $\omega$ and $\omega_2$ along small $(m-2)$-dimensional spheres $\sigma$ and $\sigma_2$. To proceed we need the following lemma.

\begin{lemma}
\label{ch2lem2}\cite{BorDr2} If $D \subset \mathbb M^{m}(c)$ is a closed $\kappa_1$-convex domain (where for $c = 0$ we assume that $\kappa_1 > 0$, for $c > 0$ we assume $\kappa_1 \geqslant 0$ and for $c < 0$ we assume that $\kappa_1 > \sqrt{-c}$), then for any two points $A, B$ from $D$ every smaller circular arc of geodesic curvature equal to $\kappa_1$ that joins $A$ and $B$ lies in the domain $D$.
\end{lemma}

Since $D$ is, in particular, a $\kappa_1$-convex domain, then by Lemma~\ref{ch2lem2} the part $\Theta$ of the envelope of this family lying between two hyperplanes $\pi$ and $\pi_2$ (corresponding to the spheres $\sigma = \pi \cap \omega$ and $\sigma_2 = \pi_2 \cap \omega_2$), lies inside $D$  (see Fig.~\ref{pic2}). 

Let $\omega^-$ and $\omega_2^+$ be the spherical caps cut from $\omega$ and $\omega_2$ by the planes $\pi$ and $\pi_2$ in a way that $\Theta$ and $\omega^-$, $\Theta$ and $\omega_2^+$ lie on the different sides with respect to $\pi$ and $\pi_2$, correspondingly. 

Consider a $C^{1,1}$ smooth hypersurface $\omega^- \cup \Theta \cup \omega_2^+$ denoted by $\Omega$. By the arguments above, $\Omega$ lie in $D$.

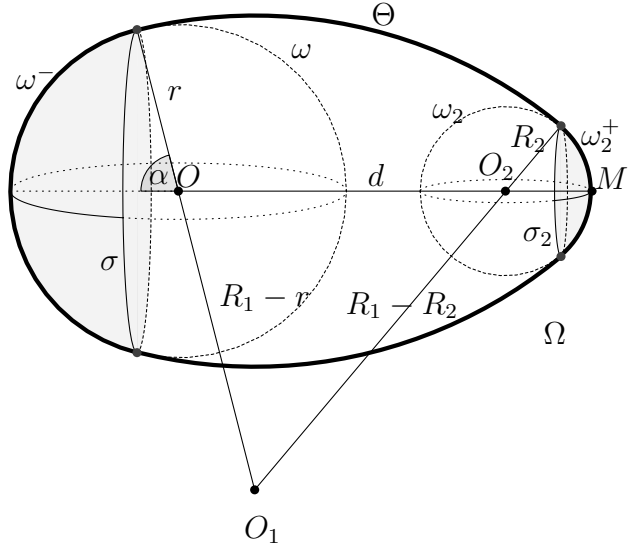
\begin{figure}[h]

\begin{center}
\definecolor{uququq}{rgb}{0.25,0.25,0.25}
\begin{tikzpicture}[line cap=round,line join=round,>=triangle 45,x=1.0cm,y=1.0cm,scale=1.7]
\clip(-1.58,-2.77) rectangle (3.51,1.63);
\draw [shift={(-0.03,0)},fill=black,fill opacity=0.1] (0,0) -- (104.17:0.29) arc (104.17:180:0.29) -- cycle;
\draw(6.18,1.11) -- (7.15,1.11);
\draw [shift={(-0.03,0)},dotted]  plot[domain=0:3.14,variable=\t]({1*1.3*cos(\t r)+0*0.22*sin(\t r)},{0*1.3*cos(\t r)+1*0.22*sin(\t r)});
\draw [shift={(2.5,0)},dotted]  plot[domain=0:3.14,variable=\t]({1*0.66*cos(\t r)+0*0.09*sin(\t r)},{0*0.66*cos(\t r)+1*0.09*sin(\t r)});
\draw [shift={(-0.35,0)}] plot[domain=0:3.14,variable=\t]({0*1.26*cos(\t r)+-1*0.11*sin(\t r)},{1*1.26*cos(\t r)+0*0.11*sin(\t r)});
\draw [shift={(-0.35,0)},dash pattern=on 1pt off 1pt,fill=black,fill opacity=0.05]  plot[domain=-3.14:0,variable=\t]({0*1.26*cos(\t r)+-1*0.11*sin(\t r)},{1*1.26*cos(\t r)+0*0.11*sin(\t r)});
\draw [shift={(2.93,0)},dash pattern=on 1pt off 1pt]  plot[domain=-3.14:0,variable=\t]({0*0.51*cos(\t r)+-1*0.05*sin(\t r)},{1*0.51*cos(\t r)+0*0.05*sin(\t r)});
\draw [shift={(2.93,0)},fill=black,fill opacity=0.05]  plot[domain=0:3.14,variable=\t]({0*0.51*cos(\t r)+-1*0.05*sin(\t r)},{1*0.51*cos(\t r)+0*0.05*sin(\t r)});
\draw [shift={(0.56,-2.33)},line width=1.6pt]  plot[domain=0.87:1.82,variable=\t]({1*3.7*cos(\t r)+0*3.7*sin(\t r)},{0*3.7*cos(\t r)+1*3.7*sin(\t r)});
\draw [shift={(0.56,2.33)},line width=1.6pt]  plot[domain=4.47:5.41,variable=\t]({1*3.7*cos(\t r)+0*3.7*sin(\t r)},{0*3.7*cos(\t r)+1*3.7*sin(\t r)});
\draw [shift={(2.5,0)},dash pattern=on 1pt off 1pt]  plot[domain=0.87:5.41,variable=\t]({1*0.66*cos(\t r)+0*0.66*sin(\t r)},{0*0.66*cos(\t r)+1*0.66*sin(\t r)});
\draw [shift={(2.5,0)},line width=1.6pt,fill=black,fill opacity=0.05]  plot[domain=-0.87:0.87,variable=\t]({1*0.66*cos(\t r)+0*0.66*sin(\t r)},{0*0.66*cos(\t r)+1*0.66*sin(\t r)});
\draw [shift={(-0.03,0)},line width=1.6pt,fill=black,fill opacity=0.05]  plot[domain=1.82:4.47,variable=\t]({1*1.3*cos(\t r)+0*1.3*sin(\t r)},{0*1.3*cos(\t r)+1*1.3*sin(\t r)});
\draw [shift={(-0.03,0)},dash pattern=on 1pt off 1pt]  plot[domain=-1.82:1.82,variable=\t]({1*1.3*cos(\t r)+0*1.3*sin(\t r)},{0*1.3*cos(\t r)+1*1.3*sin(\t r)});
\draw [shift={(-0.03,0)},line width=0.4pt]  plot[domain=3.14:4.38,variable=\t]({1*1.3*cos(\t r)+0*0.22*sin(\t r)},{0*1.3*cos(\t r)+1*0.22*sin(\t r)});
\draw [shift={(-0.03,0)},dotted]  plot[domain=-1.91:0,variable=\t]({1*1.3*cos(\t r)+0*0.22*sin(\t r)},{0*1.3*cos(\t r)+1*0.22*sin(\t r)});
\draw [shift={(2.5,0)},dotted]  plot[domain=3.14:5.31,variable=\t]({1*0.66*cos(\t r)+0*0.09*sin(\t r)},{0*0.66*cos(\t r)+1*0.09*sin(\t r)});
\draw [shift={(2.5,0)},line width=0.4pt]  plot[domain=-0.97:0,variable=\t]({1*0.66*cos(\t r)+0*0.09*sin(\t r)},{0*0.66*cos(\t r)+1*0.09*sin(\t r)});
\draw (0.56,-2.33)-- (2.93,0.51);
\draw (0.56,-2.33)-- (-0.35,1.26);
\draw (1.38,1.55) node[anchor=north west] {$\Theta$};
\draw (2.45,0.6) node[anchor=north west] {$R_2$};
\draw (1.18,-0.7) node[anchor=north west] {$R_1-R_2$};
\draw (0.2,-0.64) node[anchor=north west] {$R_1 - r$};
\draw (-1.37,1.07) node[anchor=north west] {$\omega^{-}$};
\draw (0.76,1.23) node[anchor=north west] {$\omega$};
\draw (-0.72,-0.42) node[anchor=north west] {$\sigma$};
\draw (2.55,-0.2) node[anchor=north west] {$\sigma_2$};
\draw (-0.03,0)-- (3.17,0);
\draw [dotted] (-1.33,0)-- (-0.03,0);
\draw (1.85,0.76) node[anchor=north west] {$\omega_2$};
\draw (-0.2,0.9) node[anchor=north west] {${r}$};
\draw (2.71,-0.94) node[anchor=north west] {$\Omega$};
\draw (3.0,0.65) node[anchor=north west] {$\omega_2^{+}$};
\draw (1.35,0.3) node[anchor=north west] {${d}$};
\fill [color=black] (-0.03,0) circle (1.0pt);
\draw[color=black] (0.05,0.12) node {$O$};
\fill [color=black] (3.17,0) circle (1.0pt);
\draw[color=black] (3.32,0.1) node {$M$};
\fill [color=black] (6.37,1.11) circle (1.5pt);
\draw[color=black] (6.47,1.29) node {$R_1 = 1$};
\fill [color=black] (2.5,0) circle (1.0pt);
\draw[color=black] (2.43,0.19) node {$O_2$};
\fill [color=black] (0.56,-2.33) circle (1.0pt);
\draw[color=black] (0.62,-2.64) node {$O_1$};
\fill [color=uququq] (-0.35,1.26) circle (1.0pt);
\draw[color=black] (-0.18,0.11) node {$\alpha$};
\fill [color=uququq] (2.93,0.51) circle (1.0pt);
\fill [color=uququq] (-0.35,-1.26) circle (1.0pt);
\fill [color=uququq] (2.93,-0.51) circle (1.0pt);
\end{tikzpicture}
\end{center}

\caption{The hypersurface $\Omega = \omega^- \cup \Theta \cup \omega_2^+$.}
\label{pic2}

\end{figure}

Now, using~(\ref{contr3}) we are going to show  that it is possible to inscribe inside $\Omega$ a sphere of radius $r'$ strictly greater than the radius $r$ of the inscribed sphere. This will give us  a desired contradiction, since by definition the inscribed sphere is a sphere of maximum radius lying in a domain .  

By symmetry, the center $C$ of the inscribed in $\Omega$ sphere lie on the geodesic $OO_2$. Let us consider a section of $\Omega$ by a totally geodesic two-dimensional plane $\Pi$ passing trough $OO_2$. Then the curve $\Omega'$ defined as $\Omega \cap \Pi$ consists of four parts: an arc $\omega'$ of a circle  $\omega \cap \Pi$ of radius $r$, an arc $\omega_2'$ of a circle $\omega_2 \cap \Pi$ of radius $R_2$, and two circular arcs of radius $R_1$. Denote by $O_1$ the center of one of the corresponding circles (see Fig.~\ref{pic2}). Then $|OO_1| = R_1 - r$, $|O_2 O_1| = R_1 - R_2$. Let $\alpha$ be the angle between the geodesic lines $OO_2$ and $OO_1$. Assume $\alpha \leqslant {\pi}/{2}$.

From the construction of $\tilde F$ it follows that the triangle with the side lengths $R_1 - r$, $R_1 - R_2$, and $\tilde R - R_2$ is a right triangle. Denote $\tilde R - R_2$ by $\tilde d$. From~(\ref{contr3}) we have that $d > \tilde d$. By the law of cosines from the geodesic triangle $\triangle O O_2 O_1$, we deduce:
\begin{enumerate}
\item
For $c = 0$, 
\begin{equation*}
\begin{aligned}
\left(R_1 - R_2\right)^2 &= \left(R_1 - r\right)^2 + d^2 - 2 d \left(R_1 - r\right) \cos \alpha \\
&> \left(R_1 - r\right)^2 + {\tilde{d}}^2 - 2 d \left(R_1 - r\right) \cos \alpha.
\end{aligned}
\end{equation*}
Since $\left(R_1 - R_2\right)^2 = \left(R_1 - r\right)^2 + {\tilde d}^2$, then from the computations above it follows that $\cos \alpha > 0$, thus $\alpha < {\pi}/{2}$.
\item
For $c = 1$, 
\begin{equation*}
\begin{aligned}
\cos \left(R_1 - R_2\right) &= \cos \left(R_1 - r\right) \cos d + \sin \left(R_1 - r\right) \sin d \cos \alpha \\
&< \cos \left(R_1 - r\right) \cos \tilde d + \sin \left(R_1 - r\right) \sin d \cos \alpha.
\end{aligned}
\end{equation*}
Recalling that $\cos \left(R_1 - R_2\right) = \cos \left(R_1 - r\right) \cos \tilde d$, we obtain $\cos \alpha > 0$, $\alpha < {\pi}/{2}$. 
\item
For $c = -1$,
\begin{equation*}
\begin{aligned}
\cosh \left(R_1 - R_2\right) &= \cosh \left(R_1 - r\right) \cosh d - \sinh \left(R_1 - r\right) \sinh d \cos \alpha \\
&> \cosh \left(R_1 - r\right) \cosh \tilde d + \sinh \left(R_1 - r\right) \sinh d \cos \alpha.
\end{aligned}
\end{equation*}
And since $\cosh \left(R_1 - R_2\right) = \cosh \left(R_1 - r\right) \cosh \tilde d$, we get $\alpha < {\pi}/{2}$. 
\end{enumerate}

Therefore, in all the cases the angle $\alpha$ is strictly less than ${\pi}/{2}$. Hence,  from the right triangle $\triangle OO_1C$ we have $|OO_1| > |O_1C|$. Thus, if $r'$ is the radius of the inscribed in $\Omega'$ circle, then $r'= R_1 - |O_1C|$ and $r' > R_1 - |OO_1| = r$. And since it is true for any plane $\Pi$, we come to the contradiction which proves~(\ref{maxdistrel}).  

Inequality~(\ref{maxdistrel}) is sharp since the equality is obviously attained for $\tilde F$. Lemma~\ref{REst} is proved. 

\end{proof}

Theorem~\ref{Rrestth} is a direct consequence of Lemma~\ref{REst}. Indeed, by Blaschke's rolling theorem, if $r$ is the radius of the inscribed sphere for $D$, then there exist a unique rounded $\kappa_1,\kappa_2$-convex spindle-shaped hypersurface with $\tilde r=r$. As the straightforward computations show, the right sides of inequalities~(\ref{Rresteq1})~-- (\ref{Rresteq3}) are the values of $\tilde R$ in terms of the inscribed sphere's radius $\tilde r = r$, after which Lemma~\ref{REst} implies Theorem~\ref{Rrestth} (the computations are omitted).  

In order to prove Theorems~\ref{th1} and~\ref{th2}, we should derive additional estimates for the spherical shell's width and the quotient of its radii in the case of rounded $\kappa_1,\kappa_2$-convex spindle-shaped hypersurfaces. These estimates are summarized in the following lemma.

\begin{lemma}
\label{rsshest} 
Suppose $\tilde F \subset \mathbb M^{m}(c)$ is a rounded $\kappa_1, \kappa_2$-convex spindle-shaped hypersurface, $\tilde r$ and $\tilde R$ are the radii of the inscribe and circumscribe spheres for $\tilde F$; then for the width $\tilde R - \tilde r$ estimates~(\ref{euclweq})~--~(\ref{hypweq}) hold. Moreover, when $c=0$ for the quotient ${\tilde R}/{\tilde r}$ estimate~(\ref{qeq}) holds.

\end{lemma}

\begin{proof}
Estimates~(\ref{euclweq})~--~(\ref{hypweq}) are obtained similarly to~\cite{BorDr2}. Let us show~(\ref{euclweq}). In the rest of the cases computations are similar.  

If $R_1 = {1}/{\kappa_1}$ and $R_2 = {1}/{\kappa_2}$ are, as usual, the radii of the spheres of the curvatures equal to $\kappa_1$ and $\kappa_2$, then it is easy to see that $$\tilde R = \sqrt{\left(R_1 - R_2\right)^2 - \left(R_1 - \tilde r\right)^2} + R_2.$$ Let us introduce a function
\begin{equation*}
w(\tilde r) = {\tilde R} - {\tilde r} = \sqrt{\left(R_1 - R_2\right)^2 - \left(R_1 - \tilde r\right)^2} + R_2 - \tilde r
\end{equation*}
defined for $\tilde r \in \left[R_2, R_1\right]$. By construction, $w \geqslant 0$, and $w(R_1) = w(R_2) = 0$. Hence, this function attains the global maximum on $\left(R_2, R_1\right)$. Solving the equation for the derivative $dw/d\tilde r=0$, which is a linear equation with respect to $\tilde r$, and substituting its solution in $w(\tilde r)$, we will get~(\ref{euclweq}).

Now let us prove estimate~(\ref{qeq}).

Similarly to the above, we introduce the function 
\begin{equation}
\label{ffunceq}
q(\tilde r) = \frac{\tilde R}{\tilde r} = \frac{1}{\tilde r} \left(\sqrt{\left(R_1 - R_2\right)^2 - \left(R_1 - \tilde r\right)^2} + R_2\right)
\end{equation}
defined for $\tilde r \in \left[R_2, R_1\right]$. Moreover, by construction, $q \geqslant 1$, and $q(R_1) = q(R_2) = 1$. Hence, this function attains the global maximum on $\left(R_2, R_1\right)$ unless the trivial case $R_1 = R_2$. Let us find this maximal value.

Solving the equation $dq/d\tilde r = 0$ for $\tilde r$, which simplifies to a quadratic equation, we get the root
\begin{equation*}
\label{roots}
\tilde r_0 = \frac{1}{R_1^2 + R_2^2} \left(2R_1^2 R_2 - R_2 \left(R_1 - R_2\right)  \sqrt{2 R_1 R_2}\right).
\end{equation*}
The function $q$ attains its maximum at $\tilde r_{0}$. Therefore
\begin{equation}
\label{fest}
q(\tilde r) \leqslant q(\tilde r_{0}) \,\,\text{for all $\tilde r$ from} \left[R_2, R_1 \right].
\end{equation}

Now we proceed with the computing $q(\tilde r_{0})$. It is straightforward to check that
\begin{equation*}
\left(R_1 - R_2\right)^2 - \left(R_1 - \tilde r_{0}\right)^2 = \frac{\left(R_1 - R_2\right)^2}{\left(R_1^2 + R_2^2\right)^2} \left(R_2 \left(R_1 - R_2\right) - R_1 \sqrt{2 R_1 R_2}\right)^2.
\end{equation*}
Thus, since $R_1 \sqrt{2 R_1 R_2} > R_2 \left(R_1 - R_2\right)$,
\begin{equation}
\label{numeq}
\begin{aligned}
&\sqrt{\left(R_1 - R_2\right)^2 - \left(R_1 - \tilde r_{0}\right)^2} + R_2 = \\
&=\frac{\sqrt{2 R_1 R_2}}{R_1^2 + R_2^2} \left( R_1  + \sqrt{2 R_1 R_2}\right) \left( R_1 + R_2 - \sqrt{2 R_1 R_2}\right).
\end{aligned}
\end{equation}
We can also rewrite $\tilde r_{0}$ in the same manner:
\begin{equation}
\label{denomeq}
\begin{aligned}
\tilde r_{0} 
&=\frac{\sqrt{2 R_1 R_2}}{R_1^2 + R_2^2} \left( R_2  + \sqrt{2 R_1 R_2}\right) \left( R_1 + R_2 - \sqrt{2 R_1 R_2}\right).
\end{aligned}
\end{equation}
Combining~(\ref{numeq}) and~(\ref{denomeq}), and recalling that $R_i = {1}/{\kappa_i}$, $i \in \{1,2\}$,  we get
\begin{equation}
\label{fmaxval}
q(\tilde r_{0}) = \frac{R_1 + \sqrt{2 R_1 R_2}}{R_2 + \sqrt{2 R_1 R_2}} =  \frac{\sqrt{\frac{\kappa_2}{\kappa_1}}+\sqrt{2}}{\sqrt{\frac{\kappa_1}{\kappa_2}}+\sqrt{2}}.
\end{equation}

Thereby, from~(\ref{ffunceq}),~(\ref{fest}), and~(\ref{fmaxval}), we finally obtain
$$
\frac{\tilde R}{\tilde r} \leqslant  \frac{\sqrt{\frac{\kappa_2}{\kappa_1}}+\sqrt{2}}{\sqrt{\frac{\kappa_1}{\kappa_2}}+\sqrt{2}},
$$
as desired.

The bound above is sharp and is attained for the rounded $\kappa_1, \kappa_2$-convex hypersurface with the radius of the inscribe sphere equal to $\tilde r_{0}$. Lemma~\ref{rsshest} is proved.

\end{proof}

Now, Theorems~\ref{th1} and~\ref{th2} are the direct consequences  of the comparison Lemma~\ref{REst} and Lemma~\ref{rsshest}.

Theorem~\ref{stabilitythm} is a corollary of Theorems~\ref{th1} and~\ref{th2}, and is obtained by substituting $\kappa_1=\kappa$, $\kappa_2=(1+\varepsilon)\kappa$ in the right sides of~(\ref{euclweq})~--~(\ref{hypweq}) and~(\ref{qeq}) with the subsequent standard analysis of the Taylor expansion for $\varepsilon$ of the obtained expression at $\varepsilon = 0$.


\begin{ack}
This work was partially done while the author was visiting the Centre de Recerca Matem\`atica as a participant of the Conformal Geometry and Geometric PDE's Program supported by a grant of the Clay Mathematics Institute. He would like to acknowledge both institutions for the given opportunities. The author is also grateful to prof. Manuel Ritor\'e, whose lecture unexpectedly raised questions solved in the paper, and to prof. Alexander Borisenko for many fruitful discussions and remarks on the paper.
\end{ack}

\end{document}